\documentclass{amsart}
\usepackage[left=1.25in, right=1.25in]{geometry}
\usepackage{cite}
\usepackage{amsthm,amsmath,amssymb,amsfonts}
\usepackage{algorithmic}
\usepackage{graphicx}
\usepackage{textcomp}
\def\BibTeX{{\rm B\kern-.05em{\sc i\kern-.025em b}\kern-.08em
    T\kern-.1667em\lower.7ex\hbox{E}\kern-.125emX}}
\usepackage{bm}
\usepackage{mathrsfs}
\usepackage{xcolor}
\usepackage{tikz}
\usepackage{hyperref}
\hypersetup{
	colorlinks,
	linkcolor={red!50!black},
	citecolor={blue!70!black},
	urlcolor={blue!80!black}
}
\usepackage{cleveref}
\usepackage{commath}
\usepackage{empheq}
\usepackage{tikz}
\def\checkmark{\tikz\fill[scale=0.4](0,.35) -- (.25,0) -- (1,.7) -- (.25,.15) -- cycle;}
\usepackage{enumitem}
\setlist{  
	listparindent=\parindent,
	parsep=0pt,
}
\usepackage[most]{tcolorbox}

\newcommand{\highlight}[1]{%
	\colorbox{black!5!white}{$\displaystyle#1$}}

\newtcolorbox{whiteFrameResultBeforeAfter}{
	colframe=black!20!white,
	top=2mm, bottom=3 mm, left=0mm, right=0mm,
	arc=0mm,
}

\newtheorem{assumption}{Assumption}
\newtheorem{theorem}{Theorem}
\newtheorem{corollary}{Corollary}
\newtheorem{lemma}{Lemma}
\newtheorem{remark}{Remark}
\newtheorem{proposition}{Proposition}

\makeatletter
\DeclareFontFamily{OMX}{MnSymbolE}{}
\DeclareSymbolFont{MnLargeSymbols}{OMX}{MnSymbolE}{m}{n}
\SetSymbolFont{MnLargeSymbols}{bold}{OMX}{MnSymbolE}{b}{n}
\DeclareFontShape{OMX}{MnSymbolE}{m}{n}{
    <-6>  MnSymbolE5
   <6-7>  MnSymbolE6
   <7-8>  MnSymbolE7
   <8-9>  MnSymbolE8
   <9-10> MnSymbolE9
  <10-12> MnSymbolE10
  <12->   MnSymbolE12
}{}
\DeclareFontShape{OMX}{MnSymbolE}{b}{n}{
    <-6>  MnSymbolE-Bold5
   <6-7>  MnSymbolE-Bold6
   <7-8>  MnSymbolE-Bold7
   <8-9>  MnSymbolE-Bold8
   <9-10> MnSymbolE-Bold9
  <10-12> MnSymbolE-Bold10
  <12->   MnSymbolE-Bold12
}{}

\let\llangle\@undefined
\let\rrangle\@undefined
\DeclareMathDelimiter{\llangle}{\mathopen}%
                     {MnLargeSymbols}{'164}{MnLargeSymbols}{'164}
\DeclareMathDelimiter{\rrangle}{\mathclose}%
                     {MnLargeSymbols}{'171}{MnLargeSymbols}{'171}
\makeatother

\newcommand{\id}{\text{Id}}
\newcommand{\SL}{\mathcal{SL}_{\lambda}}
\newcommand{\DL}{\mathcal{DL}_{\lambda}}
\newcommand{\TD}{\mathsf{T}_D}
\newcommand{\TN}{\mathsf{T}_N}
\newcommand{\TspaceN}{\mathbf{H}_N}
\newcommand{\TspaceD}{\mathbf{H}_D}
\newcommand{\domX}[1]{\mathbf{X}\left(#1\right)}
\newcommand{\domXlocD}[1]{\mathbf{X}_{\text{loc}}^D\left(#1\right)}
\newcommand{\domXlocN}[1]{\mathbf{X}_{\text{loc}}^N\left(#1\right)}
\newcommand{\domXloc}[1]{\mathbf{X}_{\text{loc}}\left(#1\right)}
\newcommand{\bra}[1]{\left(#1\right)}
\newcommand{\bLtwo}{\mathbf{L}^2}
\newcommand{\bLtwoloc}{\mathbf{L}^2_{\text{loc}}}

\newcommand{\dpair}[2]{\llangle #1 , #2 \rrangle}

\newcount\colveccount
\newcommand*\colvec[1]{
	\global\colveccount#1
	\begin{pmatrix}
		\colvecnext
	}
	\def\colvecnext#1{
		#1
		\global\advance\colveccount-1
		\ifnum\colveccount>0
		\\
		\expandafter\colvecnext
		\else
	\end{pmatrix}
	\fi
}

\begin{document}
\title[Spurious Resonances in Electromagnetism and Acoustics]{Spurious Resonances in Coupled Domain--Boundary Variational Formulations of Transmission Problems in Electromagnetism and Acoustics}
\author{Erick Schulz and Ralf Hiptmair}
\date{March 26, 2020}
\thanks{\textbf{Funding.} The work of Erick Schulz was supported by SNF as part of the grant 200021\_184848/1}
\thanks{Seminar in Applied Mathematics, Eidgenössische Technische Hochschule Zurich (ETH), Switzerland (e-mails: \href{mailto:erick.schulz@sam.math.ethz.ch}{erick.schulz@sam.math.ethz.ch}, \href{mailto:hiptmair@sam.math.ethz.ch}{hiptmair@sam.math.ethz.ch})}

\maketitle

\begin{abstract}
We develop a framework shedding light on common features of coupled variational formulations arising in electromagnetic scattering and acoustics. We show that spurious resonances haunting coupled domain-boundary formulations based on direct boundary integral equations of the first kind originate from the formal structure of their Calder\'on identities. Using this observation, the kernel of the coupled problem is characterized explicitly and we show that it completely vanishes under the exterior representation formula.
\end{abstract}
\smallskip\noindent
\small{\textbf{Keywords.} electromagnetic scattering, acoustic scattering, resonant frequencies, coupling}

\section{Introduction}
Transmission problems in electromagnetism and acoustics model the following typical experiment. An incident wave penetrates an object and travels inside the possibly inhomogeneous medium. Concurrently, it also scatters at its surface and propagates in the outside homogeneous region to eventually decay at infinity. Simulation of the complete phenomenon entails coupling the interior and exterior problems. A vast literature is devoted to the design of such couplings for various physical situations. Notably, the described setting is considered in \cite{hazard1996solution}, \cite{hiptmair2006stabilized}, \cite{buffa2003boundary}, \cite{hiptmair2003coupling} and \cite{schulz2019coupled}.

On the one hand, domain based variational methods offer a familiar way of modeling wave propagation in materials whose properties vary in space. The texts \cite{Hiptmair2015}, \cite{hiptmair2002finite}, \cite{assous2018mathematical} and \cite{monk2003finite} are thorough analyses for electromagnetism. Standard references such as \cite{steinbach2007numerical} and \cite{colton2012inverse} introduce the reader to the Helmholtz operator as it appears in acoustic scattering. 

On the other hand, boundary integral equations are capable of describing the behavior of the waves in unbounded homogeneous regions, because they provide valid Cauchy data that can be fed to the representation formula. Their complete derivation and properties can be found in \cite{steinbach2007numerical}, \cite{nedelec2001acoustic}, \cite{kress1989linear} and \cite{mclean2000strongly}. In the following, we consider in particular the \emph{direct} boundary integral equations of the \emph{first kind} detailed in \cite{sauter2010boundary}, \cite{buffa2003galerkin} and \cite{claeys2017first}.

Even if a transmission problem involving a Helmholtz-like operator $\mathsf{P} -\lambda\id$ has a unique solution at a given fixed frequency $\lambda\in\mathbb{C}$, the standard direct first-kind boundary integral equations obtained for the associated exterior problem is haunted by the existence of ``spurious frequencies": the kernel of the Dirichlet-to-Neumann map supplied by the first exterior Calder\'on identity is spanned by the interior Dirichlet $\lambda$-eigenfunctions of $\mathsf{P}$. Similarly, the related Neumann eigenspace corresponds to the kernel of the Neumann-to-Dirichlet map supplied by the second exterior Calder\'on identity. This issue was investigated for the electric field integral equation in \cite{christiansen2004discrete}. Eigenvalues of the Laplacian were studied in \cite{demkowicz1994asymptotic} and \cite{sauter2010boundary} from the perspective of resonant frequencies.

Unsurprisingly, this deficiency of the boundary integral equations carries over to the coupled domain-boundary variational formulations. Its consequences for the symmetric approach to the coupling problem in the context of electromagnetism (classical $\mathbf{E}$--$\mathbf{H}$ formulation) and acoustics (Helmholtz equation) were stated without proof in \cite{buffa2003galerkin} and \cite{hiptmair2006stabilized}, respectively. The development presented below is inspired by the analysis carried out in \cite[Lem. 6.1]{hiptmair2003coupling} and \cite[Prop. 3.1]{schulz2019coupled} for electromagnetism, where equivalence between domain-boundary couplings and associated transmission problems is established based on ideas from \cite[Sec. 4.3]{petersdorff1989boundary} for acoustics.

This article is motivated by our impression that the occurrence and nature of spurious resonances is presumably ``well-known in the community'', but that it is difficult to locate a systematic analysis and rigorous results in literature. We thus give in this essay a unified treatment of a few symmetric domain-boundary variational formulations for the time-harmonic solutions of transmission problems in electromagnetism and acoustics under a common framework. Particular problems are discussed in \Cref{sec: Examples}. Costabel's symmetric approach introduced in \cite{costabel1987symmetric} is generalized to allow for the mixed formulation of the interior problem. The lack of uniqueness due to resonant frequencies is shown to result from the formal structure of the Calder\'on identities. The phenomenon is thus shared by all three couplings under consideration. The kernel of the abstract coupled problem is fully characterized in section \ref{sec:resonant frequencies}.

We point out that, from a theoretical point of view, the post-processing required to recover the scattered waves in the exterior region restores uniqueness of the solutions. Indeed, the kernel of the Dirichlet-to-Neumann map vanishes under the representation formula. In practice however, the poor conditioning of the linear systems of boundary integral equations near so-called resonant frequencies leads to severe impact of round-off errors in computations and to slow convergence of iterative solvers, but while these so-called ``spurious resonant frequencies" generally cause instabilities after discretization that enforce the use of regularization strategies, their mere existence is harmless to the physical validity of the domain-boundary coupling models. This explains why classical coercive symmetric couplings remain nonetheless valuable pilot formulations for Galerkin discretization. We refer to \cite{sauter2010boundary} for an introduction to a classical approach originally suggested by Brakhage and Werner \cite{brakhage1965} to regularize the \textit{indirect} BIEs for the Helmholtz operator. We also point out that a CFIE-type stabilization procedure for the Helmholtz transmission problem is studied in \cite{hiptmair2006stabilized}, where a symmetric coupling stable for all positive frequencies is obtained. However, stabilization techniques are not the focus of this paper.

\section{Formal framework}\label{sec:formal framework}
\subsection{Notation and conventions}
Let $\Omega^-\subset\mathbb{R}^3$ be a bounded  simply connected domain with Lipschitz boundary $\Gamma:=\partial\Omega^-$. We think of $\Omega^-$ as a bounded volume occupied by an inhomogeneous object with a possibly ``rough" surface. Usually, $\Omega^-$ is assumed to be a curvilinear polyhedron. Throughout this work, we use $\Omega$ generically to denote either $\Omega^-$ or $\Omega^+:=\mathbb{R}^3\backslash\overline{\Omega^-}$. Physically, $\Omega^+$ often represents an unbounded homogeneous air region around $\Omega^-$. We let $\bLtwo\bra{\Omega}$ and $\bLtwo\bra{\Gamma}$ denote, respectively, spaces of square-integrable functions over $\Omega$ and $\Gamma$. Whenever it is possible, we use bold letters to differentiate vector-valued quantities from scalars. Capitals are often used to denote fields defined over a volume, while small characters usually refer to functions on $\Gamma$. The space of smooth fields compactly supported in $\Omega$ is written $\mathscr{D}\bra{\Omega}$. The subscript `loc' is used to extend a given function space $V$ to the larger space $V_{\text{loc}}$ comprising all functions $u$ such that $u\psi\in V$ for all $\psi\in\mathscr{D}\bra{\Omega}$. A prime will be used to indicate the dual of a space, e.g. $V'$. Duality parings are written with angular brackets, e.g. $\dpair{\cdot}{\cdot}$, but we also often allow ourselves to substitute integrals for these angular brackets when we want to emphasize $\bLtwo(\Omega)$ and $\bLtwo(\Gamma)$ as pivot spaces or highlight the analogy between the identities introduced in this formal framework and Green's classical formulas. 

We call \emph{weak differential operator matrices} the various linear operators that can be represented by a matrix of partial derivatives. We understand their arrangement in a weak sense. If no particular structure is recognized, then we must accept to define them on the Sobolev space $H^1(\Omega)$. However, in the models we consider in this work, the partial derivatives often sum up to form divergence and curl operators respectively defined on
\begin{subequations}
\begin{align}
\mathbf{H}\bra{\text{div},\Omega} &:=\{\mathbf{U}\in \bLtwo\bra{\Omega}\,\vert\,\text{div}\,\mathbf{U}\in \bLtwo\bra{\Omega}\},\\
\mathbf{H}\bra{\mathbf{curl},\Omega} &:=\{\mathbf{U}\in \bLtwo\bra{\Omega}\,\vert\,\mathbf{curl}\,\mathbf{U}\in \bLtwo\bra{\Omega}\}.
\end{align} 
\end{subequations}
The following Green's formulas can be extended to these spaces:
\begin{subequations}
\begin{gather}
\pm\int_{\Omega^{\mp}} \text{div}(\mathbf{U})\, P + \mathbf{U}\cdot\nabla P \, \dif \mathbf{x} = \int_{\Gamma} P\bra{\mathbf{U}\cdot\mathbf{n}}\label{eq:IBP div}\dif \sigma,\\
\pm\int_{\Omega^\mp}\mathbf{U}\cdot\mathbf{curl\,}(\mathbf{V}) -\mathbf{curl\,}(\mathbf{U})\cdot\mathbf{V}\dif \mathbf{x}= \int _{\Gamma}\mathbf{V} \cdot \bra{\mathbf{U}\times \mathbf{n}}\dif \sigma\label{eq:IBP curl}.
\end{gather}
\end{subequations}
Here and in the remainder of the paper, $\mathbf{n}(\mathbf{x})$ stands for the unit normal boundary vector field oriented outward from $\Omega^-$. The same notation is kept throughout this section.

\subsection{Boundary value problems}
We consider a formally self-adjoint linear weak differential operator matrix 
\begin{equation}
\mathsf{P}: \domXloc{\Omega}:=\domXloc{\mathsf{P},\Omega}\rightarrow\bLtwoloc\bra{\Omega}.
\end{equation}
In accordance with this definition, we assume that $\domXloc{\Omega}\subset\bLtwoloc\bra{\Omega}$. Ultimately, our goal is to develop variational transmission equations in which exterior problems of the form
\begin{equation}\label{eq:Helmholtz-like}
\bra{\mathsf{P} -\lambda\,\id}U^{\text{ext}} = 0\quad\text{in }\Omega^+
\end{equation}
for some fixed $\lambda\in\mathbb{C}$ are formulated using BIEs. What we have in mind for $\mathsf{P}$ is a range of important operators. A few examples that arise in the study of acoustic and electromagnetic scattering will be presented in \Cref{sec: Examples}.

The first step in the formulation of BVPs for $\mathsf{P}$ is to establish a definition  of boundary data. In the following, square brackets indicate the jump $\left[\mathsf{T}_{\bullet}\right]:=\mathsf{T}^-_{\bullet}- \mathsf{T}^+_{\bullet}$ of a trace, specified by $\bullet=D$ or $N$, over the boundary $\Gamma$. Let $\domXlocD{\Omega}$ and $\domXlocN{\Omega}$ be two subspaces of $\mathbf{L}^2(\Omega)$ such that
\begin{equation}
\domXloc{\Omega}\subset \domXlocD{\Omega}\cap \domXlocN{\Omega}.
\end{equation}

The next assumption is motivated by \cite[Sec. 3 \& 4]{mclean2000strongly}, \cite[Sec. 2]{buffa2003galerkin} and \cite[Sec. 3.1]{claeys2017first}, among others.

\begin{assumption}[Existence of traces and Green's second formula]\label{assum:traces}
 There exist two non-trivial Hilbert trace spaces of distributions $\TspaceN$ and $\TspaceD$ supported on $\Gamma$ that are dual under a pairing $\dpair{\cdot}{\cdot}_{\Gamma}$, together with continuous and \emph{surjective} linear operators
\begin{align}
\TD^\mp:\, \mathbf{X}_{\emph{loc}}^D(\Omega)\rightarrow \TspaceD,
&&\TN^\mp:\, \mathbf{X}_{\emph{loc}}^N(\Omega)\rightarrow \TspaceN, 
\end{align}
\emph{admitting continuous right inverses} and satisfying Green's second formula
\begin{equation}\label{eq:Green second formula}
\int_{\Omega^{\mp}}\mathsf{P}U\cdot V - U\cdot\mathsf{P}V\dif\mathbf{x}  = \pm\dpair{\mathsf{T}^{\mp}_N U}{\mathsf{T}^{\mp}_D V}_{\Gamma} \mp \dpair{\mathsf{T}^{\mp}_N V}{{\mathsf{T}^{\mp}_D U}_{\Gamma}}_{\Gamma},\qquad \forall\,U,V\in\mathbf{X}(\Omega^\mp).
\end{equation}
We suppose that $\mathscr{D}(\Omega^{\mp})\subset\ker(\TD^{\mp})\cap\ker(\TN^{\mp})$. Moreover, we take for granted that $\left[\TD\bra{\phi}\right]=\left[\TN\bra{\phi}\right]= 0$ whenever $\phi$ is smooth in a neighborhood of $\Gamma$.
\end{assumption}

The archetypes behind these operators are the Dirichlet and Neumann traces, but all the traces occurring in the examples presented in this paper also satisfy \Cref{assum:traces}.

\begin{remark}
Roughly speaking, the hypothesis that $\mathscr{D}\bra{\Omega}\subset\ker(\TD^{\mp})\cap\ker(\TN^{\mp})$ simply asks for the traces of functions vanishing on the boundary to vanish.
\end{remark}

 Given boundary data $g\in\TspaceD$ and $\eta\in\TspaceN$, we use the traces supplied in Assumption \ref{assum:traces} to impose boundary conditions in the statement of interior and exterior BVPs:
 \begin{subequations}
 \begin{align}
 \begin{cases}\label{eq:DP}
 \mathsf{P} U - \lambda U= 0, &\text{in }\Omega^{\mp}\\
 \mathsf{T}^{\mp}_D U = g, &\text{on }\Gamma,\\
 \text{radiation conditions at }\infty, &\text{if }\Omega=\Omega^+
 \end{cases}
 &&(\text{\emph{abstract} Dirichlet BVPs for }\mathsf{P})
 \end{align}
  \begin{align}\label{eq:NP}
 \begin{cases}
 \mathsf{P} U - \lambda U= 0, &\text{in }\Omega^{\mp}\\
 \mathsf{T}^{\mp}_N U = \eta, &\text{on }\Gamma,\\
 \text{radiation conditions at }\infty, &\text{if }\Omega=\Omega^+.
 \end{cases}
  && (\text{\emph{abstract} Neumann BVPs for }\mathsf{P})
 \end{align}
 \end{subequations}
 \begin{assumption}[Uniqueness for exterior BVPs]\label{assum:unique sol to exterior problems}
 	The solutions to the \emph{exterior} (abstract) Dirichlet and Neumann BVPs \eqref{eq:DP} and \eqref{eq:NP} posed on $\domX{\Omega^+}$ are unique.
 \end{assumption}
 
 See \cite[Thm. 9.11]{mclean2000strongly}, \cite[Thm. 6.10]{colton2012inverse}, \cite{hazard1996solution} and \cite[Cor. 3.9]{claeys2017first}.
 
 \subsection{Transmission problems}
 Let $\mathsf{P}$ be defined on $\mathbf{X}(\mathsf{P},\Omega^{\mp})\subset \mathbf{X}_{\mathsf{P}}^D(\Omega^{\mp})\cap \mathbf{X}_{\mathsf{P}}^N(\Omega^{\mp})$ such that it satisfies assumptions \ref{assum:traces} and \ref{assum:unique sol to exterior problems} for continuous and surjective traces $\mathsf{T}_{\mathsf{P},D}^{\mp}:\mathbf{X}_{\mathsf{P}}^D(\Omega^{\mp})\rightarrow \TspaceD\bra{\Gamma}$ and $\mathsf{T}_{\mathsf{P},N}^{\mp}:\mathbf{X}_{\mathsf{P}}^N(\Omega^{\mp})\rightarrow \TspaceN\bra{\Gamma}$. Further suppose that $\mathsf{L}$ is a linear differential operator defined on $\mathbf{X}(\mathsf{L},\Omega^-)\subset \mathbf{X}_{\mathsf{L}}^D(\Omega^{\mp})\cap \mathbf{X}_{\mathsf{L}}^N(\Omega^{\mp})$ that satisfies \Cref{assum:traces} for the traces $\mathsf{T}_{\mathsf{L},D}^-:\mathbf{X}_{\mathsf{L}}^D(\Omega^{\mp})\rightarrow \TspaceD(\Gamma)$ and $\mathsf{T}_{\mathsf{L},N}^-:\mathbf{X}_{\mathsf{L}}^N(\Omega^{\mp})\rightarrow \TspaceD(\Gamma)$. Notice that the trace spaces associated with the two operators are required to correspond.

We are interested in well-posed transmission problems of the form: given a source term $f\in \mathbf{L}^2(\Omega^-)$ and boundary data $\left(g,\eta\right)\in\TspaceD\times\TspaceN$, find $(U,U^{\text{ext}})\in \mathbf{X}(\mathsf{L},\Omega^-)\times \mathbf{X}(\mathsf{P},\Omega^+)$ satisfying
\begin{align}
\begin{cases}\label{eq:TP}
\mathsf{L} U= f, &\text{in }\Omega^{-}\\
\mathsf{P} U^{\text{ext}} - \lambda U^{\text{ext}}= 0, &\text{in }\Omega^{+}\\
\mathsf{T}_{\mathsf{L},D}^- U = \mathsf{T}_{\mathsf{P},D}^+U^{\text{ext}}+g, &\text{on }\Gamma,\\
\mathsf{T}_{\mathsf{L},N}^- U = \mathsf{T}_{\mathsf{P},N}^+U^{\text{ext}}+\eta, &\text{on }\Gamma,\\
\text{radiation conditions at }\infty,
\end{cases}
&&  && (\text{\emph{abstract} transmission problem})
\end{align}
cf. \cite[Eq. 2]{hiptmair2006stabilized}, \cite[Eq. 1.1]{hiptmair2003coupling}, \cite[Eq. 3-4]{schulz2019coupled}, \cite[Sec. 2]{hazard1996solution} and related literature. 

The operator $\mathsf{L}$ models propagation of waves inside the object $\Omega^-$. The later phenomenon can be described using different formulations, thus we emphasize that vector-valued functions $U\in\mathbf{X}(\mathsf{L},\Omega^-)$ need \emph{not} have the same number of entries as vector-valued functions $U^{\text{ext}}\in\mathbf{X}_{\text{loc}}(\mathsf{P},\Omega^{+})$. In other words, the number of unknowns in the interior problem may differ from the number of unknowns in the exterior problem. For instance, this occurs with mixed formulations, in which auxiliary variables increase the dimensionality of the system of equations. Nevertheless, the transmission problem \eqref{eq:TP} covers the important and common case where $\mathsf{L}=\mathsf{P}-\lambda\text{Id}$. Intuitively, it is a good heuristic to think of $\mathsf{L}$ as ``the operator $\mathsf{P}$ in which the spacial coefficients may vary in space''. See \Cref{fig: transmission problem}.

\begin{figure}
\begin{tikzpicture}
\filldraw[color=black, fill=black!10, very thick,dashed] (2,-1) rectangle (7,2);
\node (A) at (0.7, 0.3) {$\mathsf{T}^+_{\mathsf{P},N}U^{\text{ext}}$};
\node (B) at (2, 0.3) {};
\node (C) at (3.3, 0.3) {$\mathsf{T}^-_{\mathsf{L},N}U$};

\node (E) at (0.7, -0.3) {$\mathsf{T}^+_{\mathsf{P},D}U^{\text{ext}}$};
\node (F) at (2, -0.3) {};
\node (G) at (3.3, -0.3) {$\mathsf{T}^-_{\mathsf{L},D}U$};

\node (P) at (6,-0.5) {$\mathsf{L}U=F$};
\node (Q) at (6.3,1.5) {$\Omega^-$};
\node (R) at (8.7,-0.5) {$\mathsf{P}U^{\text{ext}}-\lambda U^{\text{ext}}=0$};
\node (Q) at (8.5,1.5) {$\Omega^+=\mathbb{R}^3\backslash\overline{\Omega^-}$};
\draw [-stealth](A) -- (B);
\draw [-stealth](C) -- (B);
\draw [-stealth](E) -- (F);
\draw [-stealth](G) -- (F);
\end{tikzpicture}
\caption{Depiction of the abstract transmission problem \eqref{eq:TP}. The shaded region represents a volume occupied by a scattering object.}
\label{fig: transmission problem}
\end{figure}
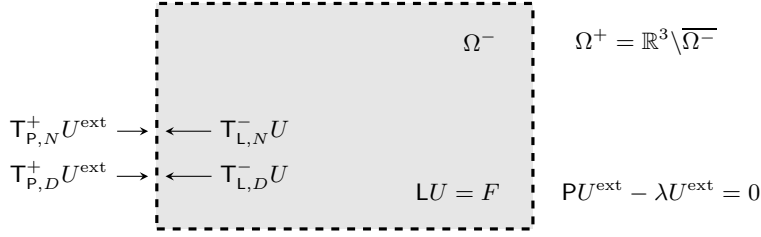

 We refer to \cite{costabel1988boundary}, \cite{monk2003finite}, \cite{buffa2003galerkin}, \cite{claeys2017first} and \cite[Sec. 3]{schulz2019coupled} for the next assumption.

\begin{assumption}[Green's first formula]\label{assum:existence of variatial interior problems}
	There exist a non-trivial subspace $\mathbf{V}\bra{\Omega^-}\subset \mathbf{X}_{\mathsf{L}}^D(\Omega^{-})$ and a continuous  bilinear form $\bm{\Phi}$ on $\mathbf{V}\bra{\Omega^-}\times\mathbf{V}\bra{\Omega^-}$ such that
	\begin{equation}\label{eq:Green first formula}
	\int_{\Omega^-}\mathsf{L}U\cdot V\dif\mathbf{x} = \bm{\Phi}\bra{U,V}+ \dpair{\mathsf{T}_{\mathsf{L},N}^{-}U}{\mathsf{T}_{\mathsf{L},D}^{-} V}_{\Gamma},\qquad\forall\,U\in\domX{\mathsf{L},\Omega^-},V\in\mathbf{V}\bra{\Omega^-}.
	\end{equation}
\end{assumption}

Assumption \ref{assum:existence of variatial interior problems} states that for $g\in \mathbf{H}_D$,
\begin{empheq}[left=\empheqlbrace]{align}
&\text{Seek } U\in \mathbf{V}(\Omega^-)\cap\{\mathsf{T}_{\mathsf{L},D}^-U = g\}\text{ satisfying}\nonumber\\
&\qquad\bm{\Phi}\bra{U,V}= 0\label{eq:IVP} \\
&\text{for all } V\in \mathbf{V}(\Omega^-)\cap\ker\bra{\mathsf{T}_{\mathsf{L},D}^{-}}.\nonumber
\end{empheq}
is a weak variational formulation for the interior Dirichlet problem
\begin{align}
\begin{cases}\label{eq:IP}
\mathsf{L} U = 0, &\text{in }\Omega^{-}\\
\mathsf{T}^{-}_{\mathsf{L},D} U = g, &\text{on }\Gamma.
\end{cases}
&&\qquad\qquad(\text{\emph{abstract} interior Dirichlet BVP for } \mathsf{L} )
\end{align}
By testing with $V\in\mathscr{D}\bra{\Omega^-}$, we immediately find that a solution $U\in\mathbf{V}\bra{\Omega^-}$ of \eqref{eq:IVP} solves $\mathsf{L}U=0$ in the sense of distributions. Therefore, it also solves \eqref{eq:IP} in $\bLtwo\bra{\Omega}$ if it is regular enough. It is necessary and sufficient that $U\in\domX{\mathsf{L},\Omega^-}$. It is thus reasonable to assume the following regularity result.

\begin{assumption}[Regularity]\label{assum:regularity}
	A distribution $U\in\mathbf{V}\bra{\Omega^-}$ which solves $\mathsf{L}U=0$ in the sense of distributions also belongs to $\mathbf{X}\bra{\mathsf{L},\Omega^-}$.
\end{assumption}
This assumption is modeled on the examples below. For e.g., in the simple case where $U\in H^1(\Omega^-)$ is a weak solution of the interior variational problem associated with the scalar Helmholtz operator, then it follows that $\nabla U\in \mathbf{H}\bra{\text{div},\Omega}$, i.e. $U\in H(\Delta,\Omega)$.

\begin{assumption}\label{assum: uniqueness transmission}
	The transmission problem \eqref{eq:TP} is uniquely solvable.
\end{assumption}

In the following sections, we use $\TD:=\mathsf{T}_{\mathsf{P},D}$ and $\TN:=\mathsf{T}_{\mathsf{P},N}$ to ease notation.
 \subsection{Representation by boundary potentials} 
 Given a formally self-adjoint weak differential operator matrix $\mathsf{L}$ and a locally integrable source term $F$, we say that $\mathsf{L} U=F$ holds in $\Omega$ \emph{in the sense of distributions} if 
 \begin{equation}\label{eq:equality in the sense of distributions}
 \dpair{\mathsf{L}U}{V} :=\int_{\Omega} U\cdot\mathsf{L}V\dif\mathbf{x} = \int_{\Omega} F\cdot V\dif\mathbf{x},\qquad \forall \,V\in\mathscr{D}(\Omega). 
 \end{equation}
From this point of view, we have $U,F\in\mathscr{D}\bra{\Omega}'$ and the action of $\mathsf{L}$ is extended by the left hand side of \eqref{eq:equality in the sense of distributions} to be also defined on the space of distributions. That is to say, the solution $U$ is interpreted as a bounded linear functional over the space of test functions.
 
Let $U\in\bLtwoloc\bra{\mathbb{R}^3}$ be such that $U\vert_{\Omega^-}\in\domX{\Omega^-}$ and $U\vert_{\Omega^+}\in\domXloc{\mathsf{P},\Omega^+}$ with $\bra{\mathsf{P} - \lambda\id}U\vert_{\Omega^{\mp}} = 0$, where the restrictions are to be understood in the sense of distributions. Using Green's second formula \eqref{eq:Green second formula} both in $\Omega^-$ and $\Omega^+$ as in \cite[Sec. 4.2]{claeys2017first}, we obtain under \Cref{assum:traces} that
 \begin{equation}
 \langle\mathsf{P}U-\lambda U,\psi\rangle = \dpair{\left[\TD U\right]}{\TN^{-} \psi}_{\Gamma} -\dpair{\TD^{-} \psi}{\left[\TN U\right]}_{\Gamma}
 \end{equation}
for all smooth compactly supported fields $\psi$ defined over $\mathbb{R}^3$. Therefore, in the sense of distributions,
\begin{equation}\label{eq:distribution formula}
\mathsf{P}U - \lambda U = \bra{\TN^-}^*\left[\TD u\right] -\bra{\TD^-}^*\left[\TN u\right],
\end{equation}
where the mappings $\bra{\TN^-}^*$ and $\bra{\TD^-}^*$ are adjoint to $\TN^-$ and $\TD^-$, respectively---to be compared with \cite[Thm. 6.10]{mclean2000strongly}, \cite[Eq. 3.8]{costabel1988boundary}, \cite[Eq. 38]{claeys2017first}.

Let $\delta_0$ be the Dirac distribution centered at $0$, i.e. $\langle\delta_0,V\rangle=V(0)$ for all $V\in\mathscr{D}(\mathbb{R}^3)$. Recall that convolution by a matrix-valued function $\mathbf{M}$ defined on $\mathbb{R}^3\backslash\{0\}$ is given by
\begin{equation}
    \mathbf{M}\star U = \int_{\mathbb{R}^3}\mathbf{M}(\mathbf{x}-\mathbf{y}) U(\mathbf{y})\dif\mathbf{x}.
\end{equation}
Compare the following assumption with \cite[Sec. 1.1.3]{sauter2010boundary}, \cite[Chap. 5]{steinbach2007numerical}, \cite[Chap. 6]{mclean2000strongly}, \cite[Sec. 4]{buffa2003galerkin} and \cite[4.1]{claeys2017first}.
\begin{assumption}[Fundamental solution]
	There exists a smooth (possibly matrix-valued) complex Green tensor $\mathbf{G}_{\lambda}$ defined over $\mathbb{R}^3\backslash\{0\}$ satisfying 
	\begin{equation}\label{fundamental solution equation}
	   \bra{\mathsf{P}-\lambda\,\emph{Id}}\mathbf{G}_{\lambda}=\delta_0\,\emph{Id}
	   	\end{equation}
	as a distribution and the radiation conditions at infinity stated in \eqref{eq:DP} and \eqref{eq:NP}.
\end{assumption}

Convolution with $\mathbf{G}_{\lambda}$ on both sides of \eqref{eq:distribution formula} using \eqref{fundamental solution equation} yields the representation formula
\begin{equation}\label{eq:representation formula}
U =  \SL\bra{\left[\TN U\right]} + \DL\bra{\left[\TD U\right]},
\end{equation}
where we have defined for all $g\in\TspaceD$ and $\eta\in\TspaceN$ the layer potentials
\begin{align}
\SL(g) := -\mathbf{G}_{\lambda}\star\bra{\bra{\TD^-}^*g},
&&\DL(\eta) := \mathbf{G}_{\lambda}\star\bra{\bra{\TN^-}^*\eta}.
\end{align}

\subsection{Boundary integral operators}
Boundary integral equations for the BVPs \eqref{eq:DP} and \eqref{eq:NP} are obtained by establishing the famous Cald\'eron identities.

The following continuity result and jump relations can be found for all examples to be covered below in \cite[Eq. 5]{hiptmair2006stabilized}, \cite[Chap. 6]{steinbach2007numerical}, \cite[Thm. 7]{buffa2003galerkin} and \cite[Thm. 5.1]{claeys2017first} (beware of the sign in the definition of the jump across $\Gamma$, which may differ from one reference to another).

\begin{assumption}[Jump identities]\label{assum:jumps}
	We assume that the boundary potentials are continuous as mappings
	\begin{align}
	\SL:\TspaceN\rightarrow\mathbf{X}_{\emph{loc}}\bra{\Omega},  &&\DL:\TspaceD\rightarrow\mathbf{X}_{\emph{loc}}\bra{\Omega},
	\end{align}
	and suppose that they satisfy the jump relations
\begin{subequations}
	\begin{align}
	\left[\TD \right]\DL &= \emph{Id},   & \left[\TD \right]\DL &= 0,\\
	\left[\TD \right]\SL &= 0,  & \left[\TN \right]\SL &= \emph{Id}.
	\end{align}
\end{subequations}
\end{assumption}
	
Applying averaged traces $\{\mathsf{T}_{\bullet}\}:=1/2\bra{\mathsf{T}_{\bullet}^+ + \mathsf{T}_{\bullet}^-}$ specified with $\bullet=D$ and $N$ to $\SL$ and $\DL$ yields four continuous boundary integral operators:
\begin{subequations}
\begin{align}
\mathcal{V}_{\lambda} := \{\TD\}\,\SL:\TspaceN\rightarrow\TspaceD,
&&\mathcal{K}^{\dag}_{\lambda} := \{\TN\}\,\SL:\TspaceN\rightarrow\TspaceN,\\
\mathcal{K}_{\lambda} := \{\TD\}\,\DL:\TspaceD\rightarrow\TspaceD,
&&\mathcal{W}_{\lambda} := \{\TN\}\,\DL:\TspaceD\rightarrow\TspaceN.
\end{align}
\end{subequations}
 Taking the traces on both sides of the representation formula \eqref{eq:representation formula} and using the jump relations of \Cref{assum:jumps}, we obtain the interior and exterior Cald\'eron identities
 \begin{subequations}
\begin{align}
\mathsf{P}U-\lambda U\text{ in }\Omega^-&&\implies&&\underbrace{
\begin{pmatrix}
\mathcal{K}_{\lambda}+\frac{1}{2}\id & \mathcal{V}_{\lambda}\\
\mathcal{W}_{\lambda}                & \mathcal{K}^{\dag}_{\lambda}+\frac{1}{2}\id
\end{pmatrix}}_{=:\mathbb{P}_{\lambda}^-}
\begin{pmatrix}
\TD^-U\\
\TN^-U
\end{pmatrix} &=
\begin{pmatrix}
\TD^-U\\
\TN^-U
\end{pmatrix},\\
\mathsf{P}U-\lambda U\text{ in }\Omega^+&&\implies&&\underbrace{
	\begin{pmatrix}
	-\mathcal{K}_{\lambda}+\frac{1}{2}\id & -\mathcal{V}_{\lambda}\\
	-\mathcal{W}_{\lambda}                & -\mathcal{K}^{\dag}_{\lambda}+\frac{1}{2}\id
	\end{pmatrix}}_{=:\mathbb{P}_{\lambda}^+}
\begin{pmatrix}
\TD^+U\\
\TN^+U
\end{pmatrix} &=
\begin{pmatrix}
\TD^+U\\
\TN^+U
\end{pmatrix},
\end{align}
\end{subequations}
respectively, cf. \cite[Sec. 6.6.]{steinbach2007numerical},\cite[Sec. 5]{buffa2003galerkin}, \cite[Sec. 5]{claeys2017first}, \cite[Sec. 2.3]{schulz2019coupled}, \cite[Sec. 3.6]{sauter2010boundary}, \cite[Sec. 4]{hiptmair2006stabilized}. Note that $\mathbb{P}_{\lambda}^+ +\mathbb{P}_{\lambda}^- = \id$ so that the range of $\mathbb{P}_{\lambda}^+$ coincides with the kernel of $\mathbb{P}_{\lambda}^-$ and vice-versa. The next theorem is a consequence of the existence of continuous right inverses for the traces stated in Assumption \ref{assum:traces}. It promotes the Cald\'eron projectors to a pivotal role in domain--boundary formulations of transmission problems.

A pair of boundary functions $(g,\eta)\in \TspaceD\times\TspaceN$ is said to be valid interior/exterior Cauchy data if there exists a distribution $U^{\mp}\in\domXloc{\Omega^{\mp}}$ solving the Dirichlet and Neumann BVPs \eqref{eq:DP} and \eqref{eq:NP} in $\Omega^{\mp}$ such that $\mathsf{T}_{\mathsf{P},D}^\mp U^{\mp}=g$ and $\mathsf{T}_{\mathsf{P},N}^\mp U^{\mp}=\eta$ . We refer to \cite[Thm. 8]{buffa2003galerkin}, \cite[Thm. 3.7]{petersdorff1989boundary} and \cite[Prop. 3.6.2]{sauter2010boundary} for the proof of the next result (cf. \cite[Lem. 6.18]{steinbach2007numerical}, \cite[Prop. 5.2]{claeys2017first} and \cite[Sec. 2.3]{schulz2019coupled}).
\begin{lemma}\label{lemma:calderon characterize cauchy data}
A pair $\left(g,\eta\right)\in\TspaceD\times\TspaceN$ is valid interior or exterior Cauchy data if and only if it lies in the kernel of $\mathbb{P}_{\lambda}^+$ or $\mathbb{P}_{\lambda}^-$, respectively.
\end{lemma}

\subsection{Boundary integral equations}
The rows of the exterior Cald\'eron identities give rise to the following two direct variational BIEs of the first kind for the exterior Dirichlet \eqref{eq:DP} and Neumann \eqref{eq:NP} problems respectively:

\begin{subequations}
\begin{empheq}[left=\empheqlbrace]{align}
&\text{Seek } \xi\in \mathbf{H}_N\bra{\Gamma}\text{ satisfying}\nonumber\\
&\qquad\int_{\Gamma}\mathcal{V}_{\lambda}\xi\cdot\zeta\dif\sigma= -\int_{\Gamma}\bra{\mathcal{K}_{\lambda}+\frac{1}{2}\id}g\cdot\zeta\dif\sigma\label{eq:DBP}\\
&\text{for all } \zeta\in \mathbf{H}_N\bra{\Gamma}.\nonumber
\end{empheq}
\begin{empheq}[left=\empheqlbrace]{align}
&\text{Seek } \xi\in \mathbf{H}_D\bra{\Gamma}\text{ satisfying}\nonumber\\
&\qquad\int_{\Gamma}\mathcal{W}_{\lambda}\xi\cdot\zeta\dif\sigma= -\int_{\Gamma}\bra{\mathcal{K}^{\dag}_{\lambda}+\frac{1}{2}\id}\eta\cdot\zeta\dif\sigma\label{eq:NBP}\\
&\text{for all } \zeta\in \mathbf{H}_D\bra{\Gamma}.\nonumber
\end{empheq}
\end{subequations}

\section{Coupled domain--boundary variational formulations}\label{sec: coupled domain-boundary var}
The idea behind the so-called symmetric approach to marrying domain and boundary variational formulations (originally developed in \cite{costabel1987symmetric} for problems involving linear strongly elliptic differential operators) is to introduce a particularly clever choice of Dirichlet-to-Neumann map 
\begin{equation}
\mathsf{DtN}:\TspaceD\rightarrow\TspaceN
\end{equation}
into Green's first formula---the validity of which, following M. Costabel, we have required in \Cref{assum:existence of variatial interior problems}. Notice that both rows of the exterior Cald\'eron projection $\mathbb{P}_{\lambda}^+$ realize a Dirichlet-to-Neumann map \cite[Sec. 3.7]{sauter2010boundary}, \cite[Sec. 4]{hiptmair2006stabilized}:
\begin{align*}
\mathsf{DtN}_1:=-\mathcal{V}_{\lambda}^{-1}\bra{\mathcal{K}_{\lambda}+\frac{1}{2}\id}:\TspaceD\rightarrow\TspaceN,
&&\mathsf{DtN}_2:=-\bra{\mathcal{K}^{\dag}_{\lambda}+\frac{1}{2}\id}^{-1}\mathcal{W}_{\lambda}:\TspaceD\rightarrow\TspaceN.
\end{align*}
M. Costabel's insight was to combine both rows into the expression
\begin{equation*}
\mathsf{DtN}:=-\mathcal{W}_{\lambda}+\bra{-\mathcal{K}^{\dag}_{\lambda}+\frac{1}{2}\id}\mathsf{DtN}_1.
\end{equation*}

Introducing the transmission conditions into \eqref{eq:Green first formula}, the transmission problem \eqref{eq:TP} can be cast into the operator equation
\begin{equation*}
\bm{\Phi} U + \bra{\mathsf{T}_{\mathsf{L},D}^{-}}^*\bra{\mathsf{DtN}\circ\mathsf{T}_{\mathsf{L},D}^-U} = \mathbf{r.h.s.},
\end{equation*}
where $\bra{\mathsf{T}_{\mathsf{L},D}^{-}}^*:\TspaceN=\TspaceD'\rightarrow\mathbf{X}^D\bra{\mathsf{L},\Omega^-}'$ denotes the adjoint of $\mathsf{T}_{\mathsf{L},D}^{-}$. Here, $\TspaceD'$ and $\TspaceN$ were identified using the duality pairing from \Cref{assum:traces}. Indeed, the details read
\begin{align}
    \mathsf{T}^-_{\mathsf{L},N}U &=  \mathsf{T}_{\mathsf{P},N}^+U^{\text{ext}}+\eta\nonumber\\
    &= -\mathcal{W}_{\lambda}\mathsf{T}^+_{\mathsf{P},D}U^{\text{ext}} + (-\mathcal{K}^{\dag}+\frac{1}{2}\text{Id}) \mathsf{T}_{\mathsf{P},N}^+U^{\text{ext}} + \eta\nonumber\\
    &=-\mathcal{W}_{\lambda}\mathsf{T}^+_{\mathsf{P},D}U^{\text{ext}} + (-\mathcal{K}^{\dag}+\frac{1}{2}\text{Id}) \mathsf{DtN}_1 (\mathsf{T}_{\mathsf{P},D}^+U^{\text{ext}})  + \eta\label{eq: auxiliary var}\\
    &= \mathsf{DtN} (\mathsf{T}_{\mathsf{L},D}^-U) + \eta - \mathsf{DtN}(g). \nonumber
\end{align}

Of course, we dispense with the explicit inverse of $\mathcal{V}_{\lambda}$ by introducing an auxiliary unknown 
\begin{equation}
\xi:=\mathsf{T}_{\mathsf{P},N}^+U^{\text{ext}}= \mathsf{DtN}_1 \mathsf{T}_{\mathsf{P},D}^+U^{\text{ext}}\in\TspaceN
\end{equation}
in \eqref{eq: auxiliary var}, seeking instead a solution pair to the following variational problem.
\begin{empheq}[left=\empheqlbrace]{align}
\text{Seek }\bra{U,\xi}\in \mathbf{V}\bra{\Omega}\times \TspaceN\text{ satisfying }\nonumber\\
\begin{split}
\bm{\Phi}\bra{U,V}+ \dpair{\bra{-\mathcal{K}^{\dag}_{\lambda}+\frac{1}{2}\id} \xi}{\mathsf{T}_{\mathsf{L},D}V}&\\
+\dpair{-\mathcal{W}_{\lambda}\mathsf{T}_{\mathsf{L},D}^-U}{\mathsf{T}_{\mathsf{L},D}V} &= \mathsf{R}_{\mathbf{V}}\bra{v},\\
\dpair{\bra{\mathcal{K}_{\lambda}+\frac{1}{2}\id}\mathsf{T}_{\mathsf{L},D}^-U}{\zeta} + \dpair{\mathcal{V}_{\lambda}\,\xi}{\zeta}&= \mathsf{R_{\mathsf{T}}}\bra{\zeta},
\end{split}\label{eq:CP}\\
\text{for all }\bra{V,\zeta}\in \mathbf{V}\bra{\Omega}\times \TspaceN.\quad\quad\quad\nonumber
\end{empheq}
A few terms were moved to the continuous functionals on the right hand sides. In particular,
\begin{subequations}
\begin{align}
&\mathsf{R}_{\mathbf{V}}\bra{V} := \int_{\Omega^-}f\cdot V\dif\mathbf{x} - \dpair{\eta}{\mathsf{T}_{\mathsf{L},D}V}-\dpair{\mathcal{W}_{\lambda}g}{\mathsf{T}_{\mathsf{L},D}V},\\
&\mathsf{R_{\mathsf{T}}}\bra{\zeta} := \dpair{\bra{\mathcal{K}_{\lambda}+\frac{1}{2}\id}g}{\zeta}.
\end{align}
\end{subequations}

An alternative (arguably simpler) derivation based on the idea of modifying the Johnson-Ned\'{e}lec coupling is presented in \cite{aurada2013classical} and \cite{schulz2019coupled}. In this work, to choice to obtain \eqref{eq:CP} based on \eqref{eq: auxiliary var} is motivated by our intention to share the details behind the derivation available in \cite[Sec. 10]{buffa2003galerkin}.

\section{Resonant frequencies}\label{sec:resonant frequencies}
We call \emph{Dirichlet or Neumann resonant frequency} any eigenvalue in the Dirichlet or Neumann spectrum
\begin{align*}
\Lambda_D\bra{\mathsf{P},\Omega^-}&:=\{\lambda\in\mathbb{C}\vert\, \exists U\in\domX{\mathsf{P},\Omega^-},\, 0\neq U \text{ solving \eqref{eq:DP} in }\Omega^-\text{ with }g=0\},\\
\Lambda_N\bra{\mathsf{P},\Omega^-}&:=\{\lambda\in\mathbb{C}\vert\, \exists U\in\domX{\mathsf{P},\Omega^-},\, 0\neq U \text{ solving \eqref{eq:NP} in }\Omega^-\text{ with }\eta=0\},
\end{align*}
respectively. Given a frequency $\lambda\in\Lambda_D$ or $\Lambda_N$, we denote the $\lambda$-eigenspaces by
\begin{align*}
E^{\lambda}_D\bra{\mathsf{P},\Omega^-} :=\{U\in\domX{\mathsf{P},\Omega^-} &\vert\, U\text{ solving \eqref{eq:DP} in }\Omega^-\text{ with }g=0\},\\
E^{\lambda}_N\bra{\mathsf{P},\Omega^-}:=\{U\in\domX{\mathsf{P},\Omega^-} &\vert\, U\text{ solving \eqref{eq:DP} in }\Omega^-\text{ with }\eta=0\},
\end{align*}
respectively.

\subsection{Kernels of first-kind direct boundary integral equations}
The nontrivial eigenfunctions in $E^{\lambda}_D\bra{\mathsf{P},\Omega^-}$ and $E^{\lambda}_N\bra{\mathsf{P},\Omega^-}$ foil uniqueness of solutions of the boundary integral problems \eqref{eq:DBP} and \eqref{eq:NBP}. The next lemmas completely characterize the kernels of the operators $\mathcal{V}_{\lambda}$ and $\mathcal{W}_{\lambda}$.

\begin{lemma}\label{lem:kernel of A}
$\ker\bra{\mathcal{V}_{\lambda}} = \mathsf{T}^-_{\mathsf{P},N}\bra{E^{\lambda}_{D}\bra{\mathsf{P},\Omega^-}}$
\end{lemma}
\begin{proof}
	$\bra{\supset}$ Suppose that $\lambda\in\Lambda_D$ and let $0\neq U\in E^{\lambda}_D\bra{\mathsf{P},\Omega^-}$. By Lemma \ref{lemma:calderon characterize cauchy data}, the valid Cauchy data $\bra{0,\mathsf{T}^-_{\mathsf{P},N}U}\in \mathbf{H}_D\times\mathbf{H}_N$ for the interior problem is in the kernel of the exterior Cald\'eron projection. The first row of the matrix equation 
	\begin{align}\label{eq:kernel of A}
	\colvec{2}{-\mathcal{V}_{\lambda}\mathsf{T}^-_{\mathsf{P},N}U}{\bra{-\mathcal{K}^{\dag}_{\lambda}+\frac{1}{2}\id}\mathsf{T}^-_{\mathsf{P},N}U}
	=\begin{pmatrix}
	-\mathcal{K}_{\lambda}+\frac{1}{2}\id & -\mathcal{V}_{\lambda}\\
	-\mathcal{W}_{\lambda}                & -\mathcal{K}^{\dag}_{\lambda}+\frac{1}{2}\id
	\end{pmatrix}\colvec{2}{0}{\mathsf{T}^-_{\mathsf{P},N}U}
	=\mathbb{P}^+_{\lambda}\colvec{2}{0}{\mathsf{T}^-_{\mathsf{P},N}U}
	=\begin{pmatrix}
		0\\
		0
	\end{pmatrix}
	\end{align}
	implies that $\mathsf{T}^-_{\mathsf{P},N}U\in\ker\bra{\mathcal{V}_{\lambda}}$.
	
	$\bra{\subset}$ If $\xi\in\mathbf{H}_N$ is such that $\mathcal{V}_{\lambda}\xi = 0$, then 
	$$
	\mathbb{P}^+_{\lambda}
	\begin{pmatrix}
	0\\
	\xi
	\end{pmatrix}
	=
	\begin{pmatrix}
	0\\
	\bra{-\mathcal{K}^{\dag}_{\lambda}+\frac{1}{2}\id}\xi
	\end{pmatrix}.
	$$
	Lemma \ref{lemma:calderon characterize cauchy data} then guarantees that $\bra{0,\bra{-\mathcal{K}^{\dag}_{\lambda}+\frac{1}{2}\id}\xi}^{\top}$ is valid Cauchy data for the exterior boundary value problem \eqref{eq:DP} in $\Omega^+$. By Assumption \ref{assum:unique sol to exterior problems}, the unique solution to the exterior Dirichlet boundary value problem \eqref{eq:DP} with $g=0$ is trivial, so it must be that $\mathcal{K}^{\dag}_{\lambda}\xi=\frac{1}{2}\xi$. Therefore, we find that \begin{equation*}
	\mathbb{P}^-_{\lambda}\colvec{2}{0}{\xi}= \begin{pmatrix}
	\mathcal{K}_{\lambda}+\frac{1}{2}\id & \mathcal{V}_{\lambda}\\
	\mathcal{W}_{\lambda}                & \mathcal{K}^{\dag}_{\lambda}+\frac{1}{2}\id
	\end{pmatrix}\colvec{2}{0}{\xi}
	=\colvec{2}{0}{\xi}.
	\end{equation*}
	We conclude relying on Lemma \ref{lemma:calderon characterize cauchy data} again that there exists $0\neq U\in E^{\lambda}_D\bra{\mathsf{P},\Omega^-}$ with $\mathsf{T}^-_{\mathsf{P},N}U=\xi$.
\end{proof}

Because of the formal symmetry in the structure of the Cald\'eron identities, we also conclude from the above proof that the kernel of $\mathcal{W}_{\lambda}$ is spanned by the Dirichlet traces of the interior Neumann eigenfunctions of $\mathsf{P}$.

\begin{lemma}\label{lem:kernel of D}
$\ker\bra{\mathcal{W}_{\lambda}} = \mathsf{T}^-_{\mathsf{P},D}\bra{E^{\lambda}_{N}\bra{\mathsf{P},\Omega^-}}$
\end{lemma}
The operators on the right-hand sides of the Dirichlet and Neumann variational boundary integral equations \eqref{eq:DBP} and \eqref{eq:NBP} display similar properties.
\begin{lemma}\label{lem:kernel of B}
$\ker\bra{-\mathcal{K}^{\dag}_{\lambda}+\frac{1}{2}\text{\emph{Id}}} = \mathsf{T}^-_{\mathsf{P},N}\bra{E^{\lambda}_{D}\bra{\mathsf{P},\Omega^-}}$
\end{lemma}
\begin{proof}
	$\bra{\subset}$ Suppose that $\lambda\in\Lambda_D$ and let $U\in E^{\lambda}_{D}\bra{\mathsf{P},\Omega^-}$. Using Theorem 1, the valid Cauchy data $\bra{0,\mathsf{T}^-_{\mathsf{P},N} U}\in\mathbf{H}_D\times\mathbf{H}_N$ belongs to the kernel of $\mathbb{P}^+_{\lambda}$. We read from \eqref{eq:kernel of A} that $\mathsf{T}^-_{\mathsf{P},N} U\in\ker\bra{-\mathcal{K}^{\dag}_{\lambda}+\frac{1}{2}\id}$.
	
	$\bra{\supset}$ If $\bra{-\mathcal{K}^{\dag}_{\lambda}+\frac{1}{2}\id}\xi = 0$, then similarly as in the proof of Lemma \ref{lem:kernel of A},
	$$
	\mathbb{P}^+_{\lambda}
	\begin{pmatrix}
	0\\
	\xi
	\end{pmatrix}
	=
	\begin{pmatrix}
	-\mathcal{V}_{\lambda}\xi\\
	0
	\end{pmatrix},
	$$
	which by Lemma \ref{lemma:calderon characterize cauchy data} shows that $\bra{\mathcal{V}_{\lambda}\xi,0}$ is valid Cauchy data for the exterior boundary value problem.  By Assumption \ref{assum:unique sol to exterior problems}, the unique solution to \eqref{eq:NP} in $\Omega^+$ with $\eta=0$ is trivial, so it must be that $\mathcal{V}_{\lambda}\xi=0$. The conclusion follows from Lemma \ref{lem:kernel of A}.
\end{proof}

The following result shouldn't come as a surprise now.
\begin{lemma}\label{lem:kernel of C}
	$\ker\bra{-\mathcal{K}_{\lambda}+\frac{1}{2}\text{\emph{Id}}} = \mathsf{T}^-_{\mathsf{P},D}\bra{E^{\lambda}_{N}\bra{\mathsf{P},\Omega^-}}$
\end{lemma}

\begin{corollary}
	A solution of the Dirichlet variational boundary integral equations \eqref{eq:DBP} is unique if and only if $\lambda\notin\Lambda_D$. 
\end{corollary}

\begin{corollary}
	A solution to the Neumann variational boundary integral equations  \eqref{eq:NBP} is unique if and only if $\lambda\notin\Lambda_N$.
\end{corollary}

\subsection{Kernel of the domain-boundary coupled variational formulation}
At this point, we are well equipped to study the kernel of the operator
\begin{equation*}
\highlight{
\mathscr{P} : =
\begin{pmatrix}
\bm{\Phi} - \bra{\mathsf{T}_{\mathsf{L},D}^-}^*\mathcal{W}_{\lambda}\circ\mathsf{T}^-_{\mathsf{L},D} & \bra{\mathsf{T}_{\mathsf{L},D}^-}^*\bra{-\mathcal{K}^{\dag}_{\lambda}+\frac{1}{2}\id}\\
\bra{\mathcal{K}_{\lambda}+\frac{1}{2}\id}\mathsf{T}^{-}_{\mathsf{L},D} & \mathcal{V}_{\lambda}
\end{pmatrix}}
\end{equation*}
arising from the variational problem \eqref{eq:CP}.

\begin{whiteFrameResultBeforeAfter}
\begin{proposition}\label{proposition:characterization}
The following are equivalent.
\begin{enumerate}
\item $\bra{U,\xi}\in\mathbf{V}(\Omega^-)\times\mathbf{H}_N$ is in the kernel of $\mathscr{P}$.
\item The pair $\bra{U,\xi}\in\mathbf{V}(\Omega^-)\times\mathbf{H}_N$ is such that
\begin{itemize} 
	\item $\mathsf{L}U=0$ in the sense of distributions,
	\item $\bra{\mathsf{T}^-_{\mathsf{L},N}U-\xi}\in\mathsf{T}^-_{\mathsf{P},N}E^{\lambda}_D\bra{\mathsf{P},\Omega^-}$,
	\item $\bra{\mathsf{T}^-_{\mathsf{L},D}U,\mathsf{T}^-_{\mathsf{L},N}U}$ is valid Cauchy data in $\Omega^+$.
\end{itemize}
\end{enumerate}
\end{proposition}
\end{whiteFrameResultBeforeAfter}
\begin{proof}
	$\bra{1 \Rightarrow 2}$ Suppose that $\bra{U,\xi}\in\mathbf{V}(\Omega^-)\times\mathbf{H}_N$ is such that for all $V\in\mathbf{V}(\Omega^-)$,
	\begin{align*}
	\bm{\Phi}(U,V) + \dpair{\bra{-\mathcal{K}^{\dag}_{\lambda}+\frac{1}{2}\id} \xi}{\mathsf{T}_{\mathsf{L},D}V}
	+\dpair{-\mathcal{W}_{\lambda}\mathsf{T}^-_{\mathsf{L},D} U}{\mathsf{T}_{\mathsf{L},D}V} &= 0,\\
	\dpair{\bra{\mathcal{K}_{\lambda}+\frac{1}{2}\id}\mathsf{T}^-_{\mathsf{L},D} U}{\zeta} + \dpair{\mathcal{V}_{\lambda}\,\xi}{\zeta}&= 0.
	\end{align*}
	There are three elements that we need to check.
	
	Testing with $V\in\mathscr{D}\bra{\Omega}$, we immediately find that $\mathsf{L}U=0$ holds in the sense of distributions \checkmark. 
	
	Therefore, we can rely on Assumption \ref{assum:regularity} and use the generalized version \eqref{eq:Green first formula} of Green's first formula to obtain
	\begin{subequations}
	\begin{align}
	 \dpair{\bra{-\mathcal{K}^{\dag}_{\lambda}+\frac{1}{2}\id} \xi}{\mathsf{T}_{\mathsf{L},D}^-V}
	+\dpair{-\mathcal{W}_{\lambda}\mathsf{T}^-_{\mathsf{L},D}U}{\mathsf{T}_{\mathsf{L},D}^-V} &= \dpair{\mathsf{T}^-_{\mathsf{L},N} U}{\mathsf{T}_{\mathsf{L},D}^-V},\label{eq:use green first formula}\\
	\dpair{\bra{\mathcal{K}_{\lambda}+\frac{1}{2}\id}\mathsf{T}^-_{\mathsf{L},D}U}{\zeta} + \dpair{\mathcal{V}_{\lambda}\,\xi}{\zeta}&= 0.\label{eq:adding TD on both sides}
	\end{align}
    \end{subequations}
	This allows us to evaluate
	\begin{align*}
	\mathbb{P}^+_{\lambda}
	\begin{pmatrix}
	\mathsf{T}_{\mathsf{L},D}^-U\\
	\xi
	\end{pmatrix}
	&= 		\begin{pmatrix}
	-\mathcal{K}_{\lambda}+\frac{1}{2}\id & -\mathcal{V}_{\lambda}\\
	-\mathcal{W}_{\lambda}                & -\mathcal{K}^{\dag}_{\lambda}+\frac{1}{2}\id
	\end{pmatrix}\begin{pmatrix}
	\mathsf{T}_{\mathsf{L},D}^-U\\
	\xi
	\end{pmatrix}
	=\begin{pmatrix}
	\mathsf{T}_{\mathsf{L},D}^-U\\
	\mathsf{T}_{\mathsf{L},N}^-U
	\end{pmatrix},
	\end{align*}
	where the last equality was obtained by subtracting $\mathsf{T}_{\mathsf{L},D}^-U$ on both sides of \eqref{eq:adding TD on both sides}. Since the range of the exterior Calder\'on projector coincides with the kernel of its interior counterpart, Lemma \ref{lemma:calderon characterize cauchy data} implies that the pair $\bra{\mathsf{T}_{\mathsf{L},D}^-U,\mathsf{T}_{\mathsf{L},N}^-U}\in \mathbf{H}_D\times\mathbf{H}_N$ is valid exterior Cauchy data for \eqref{eq:DP} in $\Omega^+$ \checkmark. 
	
	Moreover, from \eqref{eq:use green first formula} and \eqref{eq:adding TD on both sides}, we know that 
	\begin{align*}
	\mathcal{V}_{\lambda}\mathsf{T}_{\mathsf{L},D}U = - \mathcal{V}_{\lambda}\xi, &&\text{ and }
	&&\mathcal{W}_{\lambda}\mathsf{T}_{\mathsf{L},D}U = \bra{-\mathcal{K}^{\dag}_{\lambda}+\frac{1}{2}\id}\xi - \mathsf{T}_{\mathsf{L},N}U.
	\end{align*}
	Hence,
	\begin{align}
	\begin{split}
	 0 =\mathbb{P}^-_{\lambda}
	 \begin{pmatrix}
	 \mathsf{T}_{\mathsf{L},D}^-U\\
	 \mathsf{T}_{\mathsf{L},N}^-U
	 \end{pmatrix}
	 =\begin{pmatrix}
	 \mathcal{K}_{\lambda}+\frac{1}{2}\id & \mathcal{V}_{\lambda}\\
	 \mathcal{W}_{\lambda}                & \mathcal{K}^{\dag}_{\lambda}+\frac{1}{2}\id
	 \end{pmatrix}
	 \begin{pmatrix}
	 \mathsf{T}_{\mathsf{L},D}^-U\\
	 \mathsf{T}_{\mathsf{L},N}^-U
	 \end{pmatrix}
	 =\begin{pmatrix}
	 \mathcal{V}_{\lambda}\bra{\mathsf{T}_{\mathsf{L},N}^-U-\xi}\\
	\bra{-\mathcal{K}^{\dag}_{\lambda}+\frac{1}{2}\id}\bra{\xi-\mathsf{T}_{\mathsf{L},N}^-U}
	 \end{pmatrix}
	 \end{split}
	\end{align}
	and conclude from Lemma \ref{lem:kernel of A} that $\mathsf{T}_{\mathsf{L},N}^-U-\xi\in\mathsf{T}^-_{\mathsf{P},N}E^{\lambda}_D\bra{\mathsf{P},\Omega^-}$ \checkmark.
	
	$\bra{2 \Rightarrow 1}$ Since $\mathsf{T}^-_{\mathsf{L},N}U-\xi$ is the interior Neumann trace of a Dirichlet $\lambda$-eigenfunction of $\mathsf{P}$, it follows from Lemma \ref{lem:kernel of A} and Lemma \ref{lem:kernel of B} that
    $$
    \mathcal{V}_{\lambda}\mathsf{T}_N^-U = \mathcal{V}_{\lambda}\xi,
	$$
	and
	$$
	\bra{-\mathcal{K}^{\dag}_{\lambda}+\frac{1}{2}\id}\mathsf{T}_N^-U = \bra{-\mathcal{K}^{\dag}_{\lambda}+\frac{1}{2}\id}\xi.
	$$
	Moreover, because $\mathsf{L}U=0$ in the sense of distributions, then Assumption \ref{assum:regularity} guarantees that $U\in\mathbf{X}\bra{\mathsf{L},\Omega^-}$. We can thus integrate by parts using Assumption \ref{assum:existence of variatial interior problems} to verify that $\bm{\Phi}(U,V) = \dpair{-\mathsf{T}_{\mathsf{L},N}^{-}U}{\mathsf{T}^{-}_{\mathsf{L},D} V}_{\Gamma}$ for all $V\in\mathbf{V}\bra{\Omega^-}$.
	
	Therefore,
	\begin{align*}
		\dpair{\mathscr{P}\colvec{2}{U}{\xi}}{\colvec{2}{V}{\zeta}}&= \Big\langle\begin{pmatrix}
			-\mathcal{W}_{\lambda} & -\mathcal{K}^{\dag}_{\lambda}+\frac{1}{2}\id\\
			      \mathcal{K}_{\lambda}+\frac{1}{2}\id  & \mathcal{V}_{\lambda}
		\end{pmatrix}
		\colvec{2}{\mathsf{T}_{\mathsf{L},D}^-U}{\mathsf{T}_{\mathsf{L},N}^-U}-\colvec{2}{\mathsf{T}_{\mathsf{L},N}^-U}{0},\colvec{2}{\mathsf{T}_{\mathsf{L},D}^{-}V}{\zeta}\Big\rangle\\
		&=
		\Big\langle\begin{pmatrix}
			0 & \id\\
			-\id & 0
		\end{pmatrix}
		\mathbb{P}^-_{\lambda}
		\begin{pmatrix}
			\mathsf{T}_{\mathsf{L},D}^-U\\
			\mathsf{T}_{\mathsf{L},N}^-U
		\end{pmatrix}, \colvec{2}{\mathsf{T}_{\mathsf{L},D}^-V}{\zeta}\Big\rangle
	\end{align*}
	vanishes for all $\bra{V,\zeta}\in \mathbf{V}\bra{\Omega}\times \TspaceN$, since valid exterior Cauchy data for $\mathsf{P}-\lambda\id$ lies in the kernel of the interior Cald\'eron projector. This shows $\bra{U,\xi}\in\ker\bra{\mathscr{P}}$.
\end{proof}

The previous characterization is technical, but it tells us a lot more than what is immediately apparent. It leads to the following main result.

\begin{whiteFrameResultBeforeAfter}
\begin{theorem}\label{thm:uniquess of solutions CP}
	The interior function $U\in\mathbf{V}(\Omega^-)$ of a solution pair $\bra{U,\xi}\in\mathbf{V}(\Omega^-)\times\mathbf{H}_N$ solving the coupled variational problem \eqref{eq:CP} is always unique. If $\lambda\notin\Lambda_{D}$, then the boundary data $\xi$ is also unique. It is otherwise only unique up to adding a boundary function lying in $\mathsf{T}^-_{\mathsf{P},N}(E^{\lambda}_{D}(\mathsf{P},\Omega^-))$. In other words,
	\begin{equation}\label{eq:ker P script}
	\ker\bra{\mathscr{P}} = \{0\}\times \mathsf{T}^-_{\mathsf{P},N}E^{\lambda}_D\bra{\mathsf{P},\Omega^-}.
	\end{equation}
\end{theorem}
\end{whiteFrameResultBeforeAfter}
\begin{proof}
	$\bra{\subset}$ Suppose that the pair $(U,\xi)\in\mathbf{V}(\Omega^-)\times\mathbf{H}_N$ is in the kernel of $\mathscr{P}$. By Proposition \ref{proposition:characterization}, $U\in\mathbf{X}\bra{\mathsf{L},\Omega^-}$ and it solves $\mathsf{L}U = 0$ in the sense of distributions. The lemma also guarantees that the boundary field $(\mathsf{T}^-_{\mathsf{L},D}U,\mathsf{T}^-_{\mathsf{L},N}U)$ is valid Cauchy data for the Dirichlet BVP \eqref{eq:DP} in $\Omega^+$. Thus, $\exists\, U^{\text{ext}}\in \mathbf{X}(\mathsf{P},\Omega^+)$ with  $\mathsf{P} U - \lambda U= 0$ satisfying $\mathsf{T}_{\mathsf{L},D}^- U = \mathsf{T}_{\mathsf{P},D}^+U^{\text{ext}}$ and $\mathsf{T}_{\mathsf{L},N}^- U = \mathsf{T}_{\mathsf{P},N}^+U^{\text{ext}}$.
	
	The pair $(U,U^{\text{ext}})$ solves the transmission problem \eqref{eq:TP} with $g=0$ and $\eta=0$. Therefore, by Assumption \ref{assum: uniqueness transmission}, it can only be the trivial solution.
	
	In particular, $U=0$. Going back to Proposition \ref{proposition:characterization} with this new information, we are left with the assertion (using $(1)\implies (2)$) that $\xi\in\mathsf{T}^-_{\mathsf{P},N}E^{\lambda}_D(\mathsf{P},\Omega^-)$.
	
	$\bra{\supset}$ It follows immediately from $((2)\implies (1))$ in \Cref{proposition:characterization} that $(0,\xi)\in\ker(\mathscr{P})$ for all $\xi\in \mathsf{T}^-_{\mathsf{P},N}E^{\lambda}_D(\mathsf{P},\Omega^-)$, because $(0,0)$ is the valid Cauchy data associated with the trivial solution.
\end{proof}

\subsection{Recovery of field solution in $\Omega^+$}
 In practice, one is less interested by the solution pair $(U,\,\xi)$ of \eqref{eq:CP} than by the actual simulation $(U,\,U^{\text{ext}})$ solving the transmission problem \eqref{eq:TP}. To recover the exterior function $U^{\text{ext}}$, we use the exterior representation formula
\begin{equation}\label{eq:exterior representation formula}
U^{\text{ext}} =  -\SL\bra{\mathsf{T}^+_{\mathsf{P},N} U} - \DL\bra{\mathsf{T}^+_{\mathsf{P},D} U}
\end{equation}
obtained from \eqref{eq:representation formula}. This step was called \emph{post-processing} in the introduction. 

It goes as follows. The right hand side of \eqref{eq:exterior representation formula} defines an operator $\mathfrak{R}:\mathbf{H}_D\times\mathbf{H}_N\rightarrow \mathbf{X}(\mathsf{P},\Omega^+)$ by
\begin{equation*}
\mathfrak{R}\colvec{2}{h}{\zeta}=-\SL\bra{\zeta} - \DL\bra{h}
\end{equation*}
Therefore, given a solution pair $(U,\,\xi)$ solving \eqref{eq:CP}, one retrieves the value of the scattered wave at a location $\mathbf{x}\in\Omega^+$ in the exterior region by computing
\begin{equation}
U^{\text{ext}}\bra{\mathbf{x}} = \mathfrak{R}\colvec{2}{\mathcal{T}_{\mathsf{L},D}U-g}{\xi}\bra{\mathbf{x}}.
\end{equation}

Because \eqref{eq:ker P script} was established in Theorem \ref{thm:uniquess of solutions CP}, we need to verify the following.
\begin{whiteFrameResultBeforeAfter}
\begin{proposition}
	$\{0\}\times \mathsf{T}^-_{\mathsf{P},N}E^{\lambda}_D\bra{\mathsf{P},\Omega^-}\subset \ker\bra{\mathfrak{R}}$
\end{proposition}
\end{whiteFrameResultBeforeAfter}
\begin{proof}
Let $\xi\in\mathsf{T}^-_{\mathsf{P},N}E^{\lambda}_D(\mathsf{P},\Omega^-)$. Using the jump identities of Assumption \ref{assum:jumps}, we notice that
\begin{align*}
\mathsf{T}^+_{\mathsf{P},D}\,\mathfrak{R}\colvec{2}{0}{\xi} &=-\{\mathsf{T}_{\mathsf{P},D}\}\,\SL\bra{\xi}\\
&= -\mathcal{V}_{\lambda}\bra{\xi}
\end{align*}
vanishes by Lemma \ref{lem:kernel of A}. We conclude that $\mathfrak{R}\colvec{2}{0}{\xi}$ solves \eqref{eq:DP} in $\Omega^+$ with $g=0$. By assumption \ref{assum:unique sol to exterior problems}, this can only occur for $\mathfrak{R}\bra{0\,\,\xi}^{\top}=0$ in $\bLtwo(\Omega^+)$.
\end{proof}
Since $\mathfrak{R}$ is linear, this confirms uniqueness of the pair $\bra{U,\,U^{\text{ext}}}$, and along with it validity of the coupled problem \eqref{eq:CP} as a physical model for electromagnetic and acoustic transmission problems.

\section{Examples}\label{sec: Examples}
We now survey three concrete examples of transmission problems where the above assumptions are met.
\subsection{Acoustics in frequency domain}
The simplest examples of BVPs satisfying these hypotheses are obtained from elliptic operators acting on scalar real-valued functions, of which the Laplacian
\begin{equation*}\label{eq:elliptic operator}
\mathsf{P}:=-\Delta = -\,\text{div}\circ\nabla = -\sum_{i=1}^3\partial^2_i
\end{equation*}
is the most famous one. It acts on a suitably scaled pressure amplitude $U$ in the scalar Helmholtz equation
\begin{equation}\label{eq:Helmholtz}
\mathsf{L}U:=-\,\text{div}\bra{\nabla U} - \kappa^2 r(\mathbf{x}) U = 0
\end{equation}
that models the propagation of plane time harmonic sound waves with real positive wave number $\kappa>0$. While the bounded refractive index $r(\mathbf{x})$  may vary inside the inhomogeneous body $\Omega^-$, it is a constant $r_0\in\mathbb{R}$ in the unbounded air region $\Omega^+$, leading to an exterior problem involving $\mathsf{P}-\lambda$ with $\lambda=\kappa^2r_{0}$. BIEs offer the most flexible way of tackling the exterior problem, but a domain formulation is best suited to deal with the interior inhomogeneity. Because of its simplicity, acoustic scattering thus presents itself as a canonical example to illustrate the relevance of coupled domain--boundary variational formulations.

In this framework, the domain of the Laplace operator is easily seen to be 
\begin{equation}
\mathbf{X}_{\text{loc}}(\Omega):=H_{\text{loc}}\bra{\Delta,\Omega}=\{U\in H_{\text{loc}}^1(\Omega)\vert \,\nabla U\in \mathbf{H}_{\text{loc}}\bra{\text{div},\Omega}\}.
\end{equation}
Boundary value problems are stated using the classical Dirichlet and Neumann traces 
\begin{align*}
&\gamma^{\mp} U\,(x) = \lim_{\Omega^{\mp}\ni y\rightarrow x}U(y),
&&\gamma_n^{\mp}U\,(x) = -\lim_{\Omega^{\mp}\ni y\rightarrow x}\mathbf{n}(x)\cdot\nabla U\,(y),
\end{align*}
 which enter Green's identity \eqref{eq:IBP div}. These traces are well-defined on smooth scalar fields and can be extended continuously to Sobolev spaces:
 \begin{subequations} 
\begin{align}
&\mathsf{T}_{\mathsf{P},D}^\mp:=\gamma^{\mp}:H_{\text{loc}}^1\bra{\Omega^{\mp}}\rightarrow H^{1/2}\bra{\Gamma}=:\mathbf{H}_D,
&&\mathsf{T}_{\mathsf{P},N}^\mp:=\gamma_n^{\mp}: \mathbf{H}_{\text{loc}}\bra{\text{div},\Omega^\mp}\rightarrow H^{-1/2}\bra{\Gamma}=:\mathbf{H}_N.
\end{align}
\end{subequations}

The classical symmetric coupling for \eqref{eq:Helmholtz} derived in \cite{hiptmair2006stabilized} fits the abstract framework of the previous sections with $\mathsf{T}_{\mathsf{L},D}^-=\mathsf{T}_{\mathsf{P},D}^-$, $\mathsf{T}_{\mathsf{L},N}^-=\mathsf{T}_{\mathsf{P},N}^-$ and
\begin{align*}
\bm{\Phi}(U,V):= \int_{\Omega}\nabla U\cdot\nabla V - \kappa^2r(\mathbf{x})\dif\mathbf{x}
\end{align*}
defined on $\mathbf{V}(\Omega^-)\times \mathbf{V}(\Omega^-)$ where $\mathbf{V}(\Omega^-):=H^1(\Omega^-)$.

In this case, the bilinear form on the left-hand side of \eqref{eq:CP} is $H^1(\Omega^-)\times H^{-1/2}(\Gamma)$-coercive \cite[Lem. 5.1]{hiptmair2006stabilized}.

\subsection{$\mathbf{E}$--$\mathbf{H}$ electromagnetism} As explained in the introduction, we also consider the non-elliptic linear operators arising in the simulation of electromagnetic scattering phenomena. A prominent example is the $\mathbf{curl}\,\mathbf{curl}$ operator
\begin{equation}\label{eq:curl curl}
\mathbf{E}\mapsto\mathbf{curl}\bra{\,\mu^{-1}(\mathbf{x})\,\mathbf{curl}\,\mathbf{E}}
\end{equation}
occurring in the frequency domain formulation of the electric wave equation
\begin{equation}\label{eq:electric wave}
\mathsf{L}\mathbf{E}:=\mathbf{curl}\bra{\,\mu^{-1}(\mathbf{x})\,\mathbf{curl}\,\mathbf{E}} - \omega^2\epsilon(\mathbf{x})\mathbf{E} = 0,
\end{equation}
in which  $\epsilon(\mathbf{x})$ and $\mu(\mathbf{x})$ are material properties known, respectively, as the dielectric and permeability tensors. Again, these quantities are assumed constant outside the scatterer, i.e. $\mu(\mathbf{x})=\mu_0$ and $\epsilon(\mathbf{x})=\epsilon_0$ in $\Omega^+$. This is the most standard time-harmonic model for the propagation of an electromagnetic wave with angular frequency $\omega$. As opposed to the Helmholtz equation of acoustic scattering, the unknown is a vector-valued function. 

We note that the $\mathbf{curl}\,\mathbf{curl}$ operator can be represented by the operator matrix
\begin{equation}\label{eq:curlcurl in partial}
\mathsf{P} := 
\begin{pmatrix}
0 &-\partial_3 &\partial_2\\
\partial_3 &0 &-\partial_1\\
-\partial_2 & \partial_1 &0
\end{pmatrix}^2
\end{equation}
involved in the the exterior problem for $\mathsf{P}-\lambda$, where $\lambda:=\omega^2\mu_0\epsilon_0$. Its domain of definition is 
\begin{equation*}
\mathbf{X}_{\text{loc}}(\Omega):=\mathbf{H}_{\text{loc}}(\mathbf{curl}^2,\Omega):=\{\mathbf{E}\in \mathbf{H}_{\text{loc}}(\mathbf{curl}, \Omega) \,\vert\, \mathbf{curl}(\mathbf{E})\in\mathbf{H}_{\text{loc}}(\mathbf{curl}, \Omega)\}.
\end{equation*}
Well-posed boundary value problems are established for the electric wave equations by continuously extending the tangential traces
\begin{subequations}
\begin{align}
&\mathsf{T}_{\mathsf{P},D}^{\mp}:=\gamma^{\mp}_t\,\mathbf{E}(x) := \mathbf{n}(x)\times \gamma_{\tau}\bra{\mathbf{E}(\mathbf{x})},
&&\mathsf{T}_{\mathsf{P},N}^{\mp}:=\gamma_R^{\mp}\,\mathbf{E}(x) := -\gamma^{\mp}_{\tau}\mathbf{curl}\,\mathbf{E}(x)
\end{align}
\end{subequations}
to mappings
\begin{gather}
\mathsf{T}_{\mathsf{P},D}^{\mp}:\mathbf{H}_{\text{loc}}(\mathbf{curl},\Omega^\mp)\rightarrow \mathbf{H}^{-1/2}(\text{curl}_{\Gamma},\Gamma)=:\mathbf{H}_D,\\
\mathsf{T}_{\mathsf{P},L}^{\mp}:\mathbf{H}_{\text{loc}}(\mathbf{curl}^2,\Omega^\mp)\rightarrow \mathbf{H}^{-1/2}(\text{div}_{\Gamma},\Gamma)=:\mathbf{H}_N.
\end{gather}
Note that $\gamma^{\mp}_{\tau}\mathbf{E} := \mathbf{E}\times\mathbf{n}$ enters Green's identity \eqref{eq:IBP curl}. The ``magnetic trace" $\gamma^{\mp}_R\,\mathbf{E}$ plays a role akin to the Neumann trace. The relatively recent development of tangential traces theory for Lipschitz domains can be found in \cite{buffa2001traces_a}, \cite{buffa2001traces_b} and \cite{buffa2002traces}. A symmetric domain-boundary variational coupling for \eqref{eq:electric wave} fitting the framework of this article is performed in \cite{hiptmair2003coupling}. There, the bilinear form
\begin{equation*}
    \bm{\Phi}(\mathbf{E},\mathbf{V}):=\int_{\Omega}\mu^{-1}(\mathbf{x})\,\mathbf{curl}\,\mathbf{E}\cdot\mathbf{curl}\,\mathbf{V} - \omega^2\epsilon(\mathbf{x})\dif\mathbf{x}
\end{equation*}
enters Green's first formula together with the traces 
\begin{align*}
\mathsf{T}_{\mathsf{L},D}^{-} = \mathsf{T}_{\mathsf{P},D}^{-}, &&\mathsf{T}_{\mathsf{L},N}^{-}:=\gamma^{-}_R\,(\mu^{-1}(\mathbf{x})\,\mathbf{E}(x)).
\end{align*}

As proved in \cite{hiptmair2003coupling}, the bilinear form underlying the coupled variational problem \eqref{eq:CP} in this case satisfies a generalized G\r{a}ding inequality (T-coercivity) in $\mathbf{H}(\mathbf{curl},\Omega^-)\times \mathbf{H}^{-1/2}(\text{div}_{\Gamma},\Gamma)$.

\subsection{$\mathbf{A}$-$\phi$ electromagnetism} \label{sec: A-phi electromagnetism}
Equation \eqref{eq:electric wave} is obtained upon combining the dynamical equations
\begin{align*}
&\mathbf{curl}\,\mathbf{E} = -i\omega \mu(x)\mathbf{H},& &\mathbf{curl}\,\mathbf{H} = i\omega\epsilon(\mathbf{x})\mathbf{E},
\end{align*}
that are part of the $\mathbf{E}$--$\mathbf{H}$ formulation of Maxwell's equations. When the magnetic and electric fields are expressed in terms of the vector and scalar electromagnetic potentials, which satisfy $\mathbf{H}=\mu^{-1}(x)\,\mathbf{curl}\,\mathbf{A}$ and $\mathbf{E}=-\partial_t\mathbf{A}-\nabla \phi$, these two equations instead combine to form
\begin{equation*}
\mathbf{curl}\bra{\,\mu^{-1}(\mathbf{x})\,\mathbf{curl}\,\mathbf{A}} + i\omega\epsilon(\mathbf{x})\nabla\phi-\omega^2\epsilon(\mathbf{x})\mathbf{A} = 0.
\end{equation*}
Elimination of $\phi$ using the Lorentz gauge 
\begin{equation}\label{eq:Lorentz gauge}
\text{div}\bra{\epsilon(\mathbf{x})\mathbf{A}}+i\omega\phi = 0
\end{equation}
leads to the Hodge-Helmholtz equation
\begin{equation}\label{eq:Hodge-Helmholtz var}
\mathbf{curl}\bra{\,\mu^{-1}(\mathbf{x})\,\mathbf{curl}\,\mathbf{A}}\\
-\epsilon(\mathbf{x})\nabla\,\text{div}\bra{\epsilon(\mathbf{x})\mathbf{A}}-\omega^2\epsilon(\mathbf{x})\mathbf{A}=0.
\end{equation} 
\begin{remark}
The link between electromagnetism and geometry through the Hodge-Laplace operator is the subject of a vast literature. Because \eqref{eq:Hodge-Helmholtz var} is robust in the low-frequency limit $\omega\rightarrow 0$, its extension to inhomogeneous materials through the generalized Lorentz gauge \eqref{eq:Lorentz gauge} has resurfaced relatively recently as an interesting alternative to the standard electric wave equation for the simulation of some contemporary physical experiments in quantum optics \cite{chew2014vector}.
\end{remark}

When the material properties $\epsilon(\mathbf{x})=\epsilon_0$ and $\mu(\mathbf{x})=\mu_0$ are assumed constant, equation \eqref{eq:Hodge-Helmholtz var} reduces to
\begin{equation}\label{eq:Hodge-Helmholtz const}
\mathsf{P}:=\mathbf{curl}\,\mathbf{curl}\,\mathbf{A}- \eta\,\nabla\,\text{div}\,\mathbf{A} - \kappa^2\mathbf{A}= 0,
\end{equation}
where $\eta = \mu_0\epsilon_0^2$ and $\kappa^2 = \mu_0\epsilon_0\omega^2$. The domain of the so-called Hodge--Helmholtz operator on the left hand side is the intersection space $\mathbf{X}(\Omega^{\mp}):=\mathbf{H}_{\text{loc}}\bra{\mathbf{curl}^2,\Omega^{\mp}}\cap\mathbf{H}_{\text{loc}}\bra{\nabla\text{div},\Omega^{\mp}}$, where 
\begin{equation*}
\mathbf{H}_{\text{loc}}\bra{\nabla\text{div},\Omega}:=\{\mathbf{U}\in\mathbf{H}_{\text{loc}}\bra{\text{div},\Omega}\,\vert \,\text{div}\,\mathbf{U}\in H_{\text{loc}}^1(\Omega)\}.
\end{equation*}
A pair of suitable traces for the formulation of boundary value problems is given by \cite{claeys2017first,claeys2018first}
\begin{subequations}
	\begin{align}
	&\mathsf{T}^{\mp}_{\mathsf{P},N}\mathbf{A}(\mathbf{x}):=\mathcal{T}^{\mp}_\text{mg}\,\mathbf{A}(\mathbf{x})= 
	\begin{pmatrix}
	\gamma_R^{\mp}\,\mathbf{A}(\mathbf{x})\\
	\gamma_n^{\mp}\,\mathbf{A}(\mathbf{x})
	\end{pmatrix}, 
	&&\mathsf{T}^{\mp}_{\mathsf{P},D}\mathbf{A}(\mathbf{x}):=\mathcal{T}^{\mp}_\text{el}\,\mathbf{A}(\mathbf{x})= 
	\begin{pmatrix}
	\gamma_t^{\mp}\,\mathbf{A}(\mathbf{x})\\
	\eta\,\gamma^{\mp}\text{div}\mathbf{A}(\mathbf{x})
	\end{pmatrix}.
	\end{align}
\end{subequations}
Notice that their ranges are product trace spaces. This is partly due to the fact that the $\mathbf{A}$--$\phi$ potential formulation of Maxwell's equations initially introduced two unknowns in the wave equation. Going back to the Lorentz gauge \eqref{eq:Lorentz gauge}, we see in the context of transmission problems that the second component of the ``electric trace" $\mathcal{T}^{\mp}_\text{el}$ is in hiding a continuity condition for the scalar potential.  Once again, it is the ``magnetic trace" $\mathcal{T}^{\mp}_{\text{mg}}$ that resembles the Neumann trace. We refer to \cite{chew2014vector} for more details.

The natural trial and test subspaces of $\mathbf{H}\bra{\text{div},\Omega^-}\cap\mathbf{H}\bra{\mathbf{curl},\Omega^-} $ readily obtained upon establishing domain based variational formulations for \eqref{eq:Hodge-Helmholtz const} using \eqref{eq:IBP div} and \eqref{eq:IBP curl} are unfortunately not viable for discretization by finite elements \cite[Sec. 6.2]{hiptmair2002finite}.  This is the reason why in \cite{schulz2019coupled} the mixed formulation
\begin{align}
\begin{split}\label{eq:mixed formulation}
\mathbf{curl}\,\bra{\,\mu^{-1}(\mathbf{x})\,\mathbf{curl}\,\mathbf{A}}  +\epsilon(\mathbf{x})\nabla P -\omega^2\epsilon(\mathbf{x})\mathbf{A} &= \mathbf{F},\\
-\,\text{div}\,\bra{\epsilon(\mathbf{x})\mathbf{A}} - P &= 0,
\end{split}
\end{align}
is considered. The operator matrix
\begin{equation*}
\mathsf{L}:= 
\begin{pmatrix}
\mathbf{curl}\circ\mu^{-1}(\mathbf{x})\mathbf{curl}-\omega^2\epsilon(\mathbf{x}) & \epsilon(\mathbf{x})\nabla\\
-\,\text{div}(\epsilon(\mathbf{x}) \cdot) & -\,\id
\end{pmatrix}
\end{equation*}
is well defined over $\mathbf{X}(\mathsf{L},\Omega^-):=\mathbf{H}\bra{\mathbf{curl}^2,\Omega^-}\times H^1(\Omega^-)$ and therefore more convenient to model the interior problem. This subtlety justifies generalizing Green's first formula in Assumption \ref{assum:existence of variatial interior problems}, because integration by parts yields
\begin{equation}\label{eq: helmholtz green first mixed}
\int_{\Omega^-}\mathsf{L}\colvec{2}{\mathbf{A}}{P}\cdot{\colvec{2}{\mathbf{V}}{Q}}\dif\mathbf{x}=\bm{\Phi}_{\kappa}\bra{\colvec{2}{\mathbf{A}}{P},\colvec{2}{\mathbf{V}}{Q}} + \Big\langle \mathsf{T}^-_{\mathsf{L},N}\colvec{2}{\mathbf{A}}{P} ,\mathsf{T}^-_{\mathsf{L},D}\colvec{2}{\mathbf{V}}{Q}\Big\rangle, 
\end{equation}
where the bilinear form defined on $\mathbf{V}(\Omega^-)\times \mathbf{V}(\Omega^-)$ with $\mathbf{V}(\Omega^-):=\mathbf{H}(\mathbf{curl},\Omega)\times H^1(\Omega^-)$ is given by
\begin{multline*}\label{eq: sym bilinear form}
\bm{\Phi}_{\kappa}\bra{\colvec{2}{\mathbf{A}}{P},\colvec{2}{\mathbf{V}}{Q}} :=   \int_{\Omega_s}\mu^{-1}\,\mathbf{curl}\,\mathbf{A}\cdot\mathbf{curl}\,\mathbf{V}\dif\mathbf{x} + \int_{\Omega_s}\epsilon\,\nabla P\cdot\mathbf{V}\dif\mathbf{x}
- \int_{\Omega_s}P\,Q\dif\mathbf{x}\\
+ \int_{\Omega_s}\mathbf{A}\cdot \epsilon\nabla Q \dif\mathbf{x} -\omega^2\int_{\Omega_s}\epsilon\,\mathbf{A}\cdot\mathbf{V}\dif\mathbf{x}
\end{multline*}
and the traces are
\begin{align*}
    \mathsf{T}^-_{\mathsf{L},N}\colvec{2}{\mathbf{A}}{P}= 	\begin{pmatrix}
	\gamma_R^-(\mu^{-1}(\mathbf{x})\mathbf{A}(\mathbf{x}))\\
	\gamma_n^-\,(\epsilon(\mathbf{x})\mathbf{A}(\mathbf{x}))
	\end{pmatrix},
	&&
	  \mathsf{T}^-_{\mathsf{L},D}\colvec{2}{\mathbf{V}}{Q}= 	\begin{pmatrix}
	\gamma_t^-\mathbf{V}(\mathbf{x})\\
	-\gamma^-\,Q
	\end{pmatrix}.
\end{align*}

For this formulation, T-coercivity in $\mathbf{H}(\mathbf{curl},\Omega^-)\times H^1(\Omega^-)\times\mathbf{H}^{-1/2}(\text{div}_{\Gamma})\times H^{-1/2}(\Gamma)$ of the bilinear form in \eqref{eq:CP} is established in \cite[Thm. 5.6]{schulz2019coupled}.

\section{Conclusion}
We have abstracted the common characteristics of the three particular problems presented in \Cref{sec: Examples}. As a consequence, Costabel's original symmetric coupling was generalized to allow for a larger class of operators. The issues raised by spurious resonant frequencies were found to be rooted in the formal structure detailed by the framework of section \ref{sec:formal framework}. In section \ref{sec:resonant frequencies}, the consequences of their existence were investigated. In doing so, the kernels of the operators entering the problems \eqref{eq:DBP}, \eqref{eq:NBP} and \eqref{eq:CP} were completely characterized. It was also shown that the Neumann eigenfunctions which thwart the uniqueness of solutions for the coupled problem vanish under the exterior representation formula, thus showing that the complete field solution $U$ remains unique despite the existence of spurious resonance frequencies. The symmetric approach to Symmetric domain-boundary coupling therefore remains a valuable starting point for Galerkin discretization.

\bibliographystyle{plain}
\bibliography{bibliography}
\end{document}